\documentclass[11pt]{article}
\usepackage{amssymb,amsmath,amsfonts}

\newtheorem{theorem}{Theorem}[section]

\newtheorem{corollary}[theorem]{Corollary}

\newtheorem{definition}[theorem]{Definition}

\numberwithin{equation}{section}
\newenvironment{proof}{\par\noindent{\bf Proof.}}{$\square$\par\bigskip}
\input{tcilatex}
\begin{document}

\title{\textbf{Bi-periodic Fibonacci matrix polynomial and its binomial
transforms}}
\author{\texttt{Arzu Coskun} and \texttt{Necati Taskara}}
\date{Department of Mathematics, Faculty of Science,\\
Selcuk University, Campus, 42075, Konya - Turkey \\
[0.3cm] \textit{arzucoskun58@gmail.com} and \textit{ntaskara@selcuk.edu.tr}}
\maketitle

\begin{abstract}
In this paper, we consider the matrix polynomial obtained by using
bi-periodic Fibonacci matrix polynomial. Then, we give some properties and
binomial transforms of the new matrix polynomials.


\end{abstract}

\section{Introduction and Preliminaries}

The bi-periodic Fibonacci $\left\{ q_{n}\right\} _{n\in \mathbb{N}}$
sequence is defined by%
\begin{equation}
q_{n}=\left\{ 
\begin{array}{c}
aq_{n-1}+q_{n-2},\ \ \text{if }n\text{ is even} \\ 
bq_{n-1}+q_{n-2},\ \ \text{if }n\text{ is odd}\ 
\end{array}%
\right. \ ,  \label{1.1}
\end{equation}%
where $q_{0}=0,\ q_{1}=1$ and $a,b$ are nonzero real numbers$.$

Also, the bi-periodic Fibonacci $\left\{ \mathcal{F}_{n}\right\} _{n\in 
\mathbb{N}}$ matrix sequence is given as%
\begin{equation}
\mathcal{F}_{n}=\left( 
\begin{array}{cc}
\left( \frac{b}{a}\right) ^{\varepsilon (n)}q_{n+1} & \frac{b}{a}q_{n} \\ 
q_{n} & \left( \frac{b}{a}\right) ^{\varepsilon (n)}q_{n-1}%
\end{array}%
\right) \,,  \label{1.2}
\end{equation}%
where $a,b$ are nonzero real numbers and%
\begin{equation}
\varepsilon (n)=\left\{ 
\begin{array}{c}
1,\ n\text{ odd} \\ 
0,\ n\text{ even}%
\end{array}%
\right. .  \label{epsilon}
\end{equation}

In addition to these sequences, the other sequences appear in many branches
of science and have attracted the attention of mathematicians (see \cite%
{Bilgici}-\cite{Edson},\cite{Horadam}-\cite{YazlikTaskara} and the
references cited therein).

Also, the polynomials have attracted the attention of some mathematicians 
\cite{Plaza,Hoggatt,YilmazCoskunTaskara}. In \cite{YilmazCoskunTaskara}, the
authors gave the bi-periodic Fibonacci polynomial as%
\begin{equation}
q_{n}\left( a,b,x\right) =\left\{ 
\begin{array}{c}
axq_{n-1}\left( a,b,x\right) +q_{n-2}\left( a,b,x\right) ,\ \ \text{if }n%
\text{ is even} \\ 
bxq_{n-1}\left( a,b,x\right) +q_{n-2}\left( a,b,x\right) ,\ \ \text{if }n%
\text{ is odd}\ 
\end{array}%
\right.  \label{1.3}
\end{equation}%
which $q_{0}\left( a,b,x\right) =0,\ q_{1}\left( a,b,x\right) =1$ and $a,b$
are nonzero real numbers and they obtained some properties of this
polynomial. Hoggatt and Bicknell, in \cite{Hoggatt}, defined the Fibonacci,
Tribonacci, Quadranacci, $r$-bonacci polynomials. They generalized Fibonacci
polynomials and their\ relationship to diagonals of Pascal's triangle. In 
\cite{Plaza}, they give $k$-Fibonacci polynomials and offered the
derivatives of these polynomials in the form of convolution of $k$-Fibonacci
polynomials.

While on the one hand the sequences and polynomials was defined, on the
other hand it was introduced some transorms for the given sequences.
Binomial transform, $k$-Binomial transform, rising and fallling binomial
transforms are a few of these transforms (see \cite%
{FalconPlaza,YilmazTaskara}).

In this study, firstly, we introduce bi-periodic Fibonacci matrix polynomial
and give some properties of this polynomial. In Section 3, we have the new
matrix polynomial by using bi-periodic Fibonacci matrix polynomial. And, we
get the binomial, $k$-binomial, rising and falling\ transforms for the
matrix polynomial as the first time in the literature. Then, we give the
recurrence relations, generating functions and Binet formulas for these
generalized Binomial transforms.

\section{The bi-periodic Fibonacci matrix polynomial}

In this section, we focus on the bi-periodic matrix polynomial and give some
properties of this generalized polynomial. Hence, we firstly define the
bi-periodic Fibonacci matrix polynomials.

\begin{definition}
\label{def2.1} For $n\in \mathbb{N}$ and any two nonzero real numbers $a,b,$
the bi-periodic Fibonacci matrix polynomial $\left( \mathcal{F}_{n}\left(
a,b,x\right) \right) $ is defined by%
\begin{equation}
\mathcal{F}_{n}\left( a,b,x\right) =\left\{ 
\begin{array}{c}
ax\mathcal{F}_{n-1}+\mathcal{F}_{n-2},\ \ \text{if }n\text{ is even} \\ 
bx\mathcal{F}_{n-1}+\mathcal{F}_{n-2},\ \ \text{if }n\text{ is odd}\ 
\end{array}%
\right.   \label{2.1}
\end{equation}%
with initial conditions $\mathcal{F}_{0}\left( a,b,x\right) =\left( 
\begin{array}{cc}
1 & 0 \\ 
0 & 1%
\end{array}%
\right) $ and $\mathcal{F}_{1}\left( a,b,x\right) =\left( 
\begin{array}{cc}
bx & \frac{b}{a} \\ 
1 & 0%
\end{array}%
\right) .$
\end{definition}

\vskip0.4cm

In Definition $\ref{def2.1},$ the matrix $\mathcal{F}_{1}$ is analogue to
the Fibonacci $Q$-matrix which exists for Fibonacci numbers.

\begin{theorem}
\label{teo2.2}Let $\mathcal{F}_{n}\left( a,b,x\right) $ be as in (\ref{2.1}%
). Then the following equalities are valid for all positive integers:
\end{theorem}

\begin{enumerate}
\item[$(i)$] $\mathcal{F}_{n}\left( a,b,x\right) =\left( 
\begin{array}{cc}
\left( \frac{b}{a}\right) ^{\varepsilon (n)}q_{n+1}\left( a,b,x\right) & 
\frac{b}{a}q_{n}\left( a,b,x\right) \\ 
q_{n}\left( a,b,x\right) & \left( \frac{b}{a}\right) ^{\varepsilon
(n)}q_{n-1}\left( a,b,x\right)%
\end{array}%
\right) ,$

\item[$(ii)$] $\det (\mathcal{F}_{n}\left( a,b,x\right) )=\left( -\frac{b}{a}%
\right) ^{\varepsilon (n)},$
\end{enumerate}

where $q_{n}\left( a,b,x\right) $ is $n$th bi-periodic Fibonacci polynomial.

\begin{proof}
By using the iteration, it can be obtained the desired results.
\end{proof}

We obtained the Cassini identity for bi-periodic Fibonacci polynomials \cite%
{YilmazCoskunTaskara}. Using the determinant of $\mathcal{F}_{n}\left(
a,b,x\right) $ in Theorem \ref{teo2.2}, again we get%
\begin{equation*}
a^{1-\varepsilon (n)}b^{\varepsilon (n)}q_{n+1}\left( a,b,x\right)
q_{n-1}\left( a,b,x\right) -a^{\varepsilon (n)}b^{1-\varepsilon
(n)}q_{n}^{2}\left( a,b,x\right) =a\left( -1\right) ^{n}.
\end{equation*}

\begin{theorem}
\label{teo2.3} For bi-periodic Fibonacci matrix polynomial, we have the
generating function%
\begin{equation*}
\sum\limits_{i=0}^{\infty }\mathcal{F}_{i}\left( a,b,x\right) t^{i}=\dfrac{1%
}{1-\left( abx^{2}+2\right) t^{2}+t^{4}}\left( 
\begin{array}{cc}
1+bxt-t^{2} & \frac{b}{a}t+bxt^{2}-\frac{b}{a}t^{3} \\ 
t+axt^{2}-t^{3} & 1-(abx^{2}+1)t^{2}+bxt^{3}%
\end{array}%
\right) .
\end{equation*}
\end{theorem}

\begin{proof}
Assume that $G(t)$ is the generating function for the polynomial $\left\{ 
\mathcal{F}_{n}\left( a,b,x\right) \right\} _{n\in \mathbb{N}}$. Then, we
can write%
\begin{eqnarray*}
\left( 1-bxt-t^{2}\right) G\left( t\right) &=&\mathcal{F}_{0}\left(
a,b,x\right) +t\left( \mathcal{F}_{1}\left( a,b,x\right) -bx\mathcal{F}%
_{0}\left( a,b,x\right) \right) \\
&&+\sum\limits_{i=2}^{\infty }\left( \mathcal{F}_{i}\left( a,b,x\right) -bx%
\mathcal{F}_{i-1}\left( a,b,x\right) -\mathcal{F}_{i-2}\left( a,b,x\right)
\right) t^{i}.
\end{eqnarray*}

Since $\mathcal{F}_{2i+1}\left( a,b,x\right) =bx\mathcal{F}_{2i}\left(
a,b,x\right) +\mathcal{F}_{2i-1}\left( a,b,x\right) $, we get%
\begin{eqnarray*}
\left( 1-bxt-t^{2}\right) G\left( t\right) &=&\mathcal{F}_{0}\left(
a,b,x\right) +t\left( \mathcal{F}_{1}\left( a,b,x\right) -bx\mathcal{F}%
_{0}\left( a,b,x\right) \right) \\
&&+\sum\limits_{i=1}^{\infty }\left( \mathcal{F}_{2i}\left( a,b,x\right) -bx%
\mathcal{F}_{2i-1}\left( a,b,x\right) -\mathcal{F}_{2i-2}\left( a,b,x\right)
\right) t^{2i} \\
&=&\mathcal{F}_{0}\left( a,b,x\right) +t\left( \mathcal{F}_{1}\left(
a,b,x\right) -bx\mathcal{F}_{0}\left( a,b,x\right) \right) \\
&&+\left( a-b\right) xt\sum\limits_{i=1}^{\infty }\mathcal{F}_{2i-1}\left(
a,b,x\right) t^{2i-1}.
\end{eqnarray*}

Now, let%
\begin{equation*}
g(t)=\sum\limits_{i=1}^{\infty }\mathcal{F}_{2i-1}\left( a,b,x\right)
t^{2i-1}.
\end{equation*}

Since%
\begin{eqnarray*}
\mathcal{F}_{2i+1}\left( a,b,x\right) &=&bx\mathcal{F}_{2i}\left(
a,b,x\right) +\mathcal{F}_{2i-1}\left( a,b,x\right) \\
&=&bx(ax\mathcal{F}_{2i-1}\left( a,b,x\right) +\mathcal{F}_{2i-2}\left(
a,b,x\right) )+\mathcal{F}_{2i-1}\left( a,b,x\right) \\
&=&(abx^{2}+1)\mathcal{F}_{2i-1}\left( a,b,x\right) +bx\mathcal{F}%
_{2i-2}\left( a,b,x\right) \\
&=&(abx^{2}+1)\mathcal{F}_{2i-1}\left( a,b,x\right) +\mathcal{F}%
_{2i-1}\left( a,b,x\right) -\mathcal{F}_{2i-3}\left( a,b,x\right) \\
&=&(abx^{2}+2)\mathcal{F}_{2i-1}\left( a,b,x\right) -\mathcal{F}%
_{2i-3}\left( a,b,x\right) ,
\end{eqnarray*}

we have%
\begin{eqnarray*}
\left( 1-\left( abx^{2}+2\right) t^{2}+t^{4}\right) g\left( t\right) &=&%
\mathcal{F}_{1}\left( a,b,x\right) t+\mathcal{F}_{3}\left( a,b,x\right)
t^{3}-(abx^{2}+2)\mathcal{F}_{1}\left( a,b,x\right) t^{3} \\
&&+\sum\limits_{i=3}^{\infty }\left\{ 
\begin{array}{c}
\mathcal{F}_{2i-1}\left( a,b,x\right) -(abx^{2}+2)\mathcal{F}_{2i-3}\left(
a,b,x\right) \\ 
\text{ \ \ \ \ \ }+\mathcal{F}_{2i-5}\left( a,b,x\right)%
\end{array}%
\right\} t^{2i-1}.
\end{eqnarray*}

Therefore,%
\begin{eqnarray*}
g\left( t\right) &=&\frac{\mathcal{F}_{1}\left( a,b,x\right) t+\mathcal{F}%
_{3}\left( a,b,x\right) t^{3}-(abx^{2}+2)\mathcal{F}_{1}\left( a,b,x\right)
t^{3}}{1-\left( abx^{2}+2\right) t^{2}+t^{4}} \\
&=&\frac{\mathcal{F}_{1}\left( a,b,x\right) t+(bx\mathcal{F}_{0}\left(
a,b,x\right) -\mathcal{F}_{1}\left( a,b,x\right) )t^{3}}{1-\left(
abx^{2}+2\right) t^{2}+t^{4}}
\end{eqnarray*}%
and as a result, we get%
\begin{equation*}
G\left( t\right) =\frac{%
\begin{array}{c}
\mathcal{F}_{0}\left( a,b,x\right) +t\mathcal{F}_{1}\left( a,b,x\right)
+t^{2}\left( ax\mathcal{F}_{1}\left( a,b,x\right) -\mathcal{F}_{0}\left(
a,b,x\right) -abx^{2}\mathcal{F}_{0}\left( a,b,x\right) \right) \\ 
\text{ \ \ \ \ \ \ \ \ \ \ \ \ \ }+t^{3}\left( bx\mathcal{F}_{0}\left(
a,b,x\right) -\mathcal{F}_{1}\left( a,b,x\right) \right)%
\end{array}%
}{1-\left( abx^{2}+2\right) t^{2}+t^{4}}.
\end{equation*}%
which is desired equality.
\end{proof}

\begin{theorem}
\label{teo2.4} For every $n\in 
\mathbb{N}
,$ we write the Binet formula for the bi-periodic Fibonacci matrix
polynomial as the form%
\begin{equation*}
\mathcal{F}_{n}\left( a,b,x\right) =A_{1}\left( \alpha ^{n}-\beta
^{n}\right) +B_{1}\left( \alpha ^{2\left\lfloor \frac{n}{2}\right\rfloor
+2}-\beta ^{2\left\lfloor \frac{n}{2}\right\rfloor +2}\right) ,
\end{equation*}%
where%
\begin{eqnarray*}
A_{1} &=&\dfrac{\left\{ {\mathcal{F}}_{1}\left( a,b,x\right) -bx\mathcal{F}%
_{0}\left( a,b,x\right) \right\} ^{\varepsilon (n)}\left\{ 
\begin{array}{c}
ax{\mathcal{F}}_{1}\left( a,b,x\right) -\mathcal{F}_{0}\left( a,b,x\right)
\\ 
-abx^{2}\mathcal{F}_{0}\left( a,b,x\right)%
\end{array}%
\right\} ^{1-\varepsilon (n)}}{\left( ab\right) ^{\left\lfloor \frac{n}{2}%
\right\rfloor }\left( \alpha -\beta \right) x^{2\left\lfloor \frac{n}{2}%
\right\rfloor }},~ \\
B_{1} &=&\dfrac{\left( b\right) ^{\varepsilon (n)}\mathcal{F}_{0}\left(
a,b,x\right) }{\left( ab\right) ^{\left\lfloor \frac{n}{2}\right\rfloor
+1}\left( \alpha -\beta \right) x^{n+2\varepsilon (n+1)}},\text{ \ \ }%
\varepsilon (n)=n-2\left\lfloor \frac{n}{2}\right\rfloor
\end{eqnarray*}%
and $\alpha ,\beta $ are roots of $r^{2}-abx^{2}r-abx^{2}=0$ equation.
\end{theorem}

\begin{proof}
Using the partial fraction decomposition, we can rewrite $G\left( t\right) $
as%
\begin{equation*}
G\left( t\right) =\frac{1}{\alpha -\beta }\left\{ 
\begin{array}{c}
\frac{%
\begin{array}{c}
t\left\{ \beta \left( \mathcal{F}_{1}\left( a,b,x\right) -bx\mathcal{F}%
_{0}\left( a,b,x\right) \right) -bx\mathcal{F}_{0}\left( a,b,x\right)
\right\} \\ 
\text{ \ \ \ \ \ }+\beta \left( abx^{2}\mathcal{F}_{0}\left( a,b,x\right) +%
\mathcal{F}_{0}\left( a,b,x\right) -ax\mathcal{F}_{1}\left( a,b,x\right)
\right) \\ 
+abx^{2}\mathcal{F}_{0}\left( a,b,x\right) -ax\mathcal{F}_{1}\left(
a,b,x\right)%
\end{array}%
}{t^{2}-\left( \beta +1\right) }\text{ } \\ 
\text{\ \ \ } \\ 
+\frac{%
\begin{array}{c}
t\left\{ \alpha \left( bx\mathcal{F}_{0}\left( a,b,x\right) -\mathcal{F}%
_{1}\left( a,b,x\right) \right) +bx\mathcal{F}_{0}\left( a,b,x\right)
\right\} \\ 
+\alpha \left( ax\mathcal{F}_{1}\left( a,b,x\right) -\mathcal{F}_{0}\left(
a,b,x\right) -abx^{2}\mathcal{F}_{0}\left( a,b,x\right) \right) \\ 
+ax\mathcal{F}_{1}\left( a,b,x\right) -abx^{2}\mathcal{F}_{0}\left(
a,b,x\right)%
\end{array}%
}{t^{2}-\left( \alpha +1\right) }%
\end{array}%
\right\} .
\end{equation*}

Since the Maclaurin series expansion of the function $\frac{A-Bt}{t^{2}-C}$
is given by%
\begin{equation*}
\frac{A-Bt}{t^{2}-C}=\overset{\infty }{\underset{n=0}{\sum }}%
BC^{-n-1}t^{2n+1}-\overset{\infty }{\underset{n=0}{\sum }}AC^{-n-1}t^{2n},
\end{equation*}

the generating function $G\left( t\right) $ can also be expressed as%
\begin{equation*}
G(t)=\frac{1}{\alpha -\beta }\left\{ 
\begin{array}{c}
\overset{\infty }{\underset{n=0}{\sum }}\frac{%
\begin{array}{c}
\left\{ 
\begin{array}{c}
\alpha \left( \mathcal{F}_{1}\left( a,b,x\right) -bx\mathcal{F}_{0}\left(
a,b,x\right) \right) \\ 
-bx\mathcal{F}_{0}\left( a,b,x\right)%
\end{array}%
\right\} \left( \beta +1\right) ^{n+1}\text{ \ \ \ \ \ \ \ \ } \\ 
+\left\{ 
\begin{array}{c}
\beta \left( bx\mathcal{F}_{0}\left( a,b,x\right) -\mathcal{F}_{1}\left(
a,b,x\right) \right) \\ 
+bx\mathcal{F}_{0}\left( a,b\right)%
\end{array}%
\right\} \left( \alpha +1\right) ^{n+1}%
\end{array}%
}{\left( \alpha +1\right) ^{n+1}\left( \beta +1\right) ^{n+1}}t^{2n+1}\text{
\ \ \ \ \ \ \ \ \ \ \ \ } \\ 
\\ 
-\overset{\infty }{\underset{n=0}{\sum }}\frac{%
\begin{array}{c}
\alpha \left( ax\mathcal{F}_{1}\left( a,b,x\right) -\mathcal{F}_{0}\left(
a,b,x\right) -abx^{2}\mathcal{F}_{0}\left( a,b,x\right) \right) \\ 
+ax\mathcal{F}_{1}\left( a,b,x\right) -abx^{2}\mathcal{F}_{0}\left(
a,b,x\right)%
\end{array}%
}{\left( \alpha +1\right) ^{n+1}\left( \beta +1\right) ^{n+1}}\left( \beta
+1\right) ^{n+1}t^{2n} \\ 
\\ 
-\overset{\infty }{\underset{n=0}{\sum }}\frac{%
\begin{array}{c}
\beta \left( abx^{2}\mathcal{F}_{0}\left( a,b,x\right) +\mathcal{F}%
_{0}\left( a,b,x\right) -ax\mathcal{F}_{1}\left( a,b,x\right) \right) \\ 
+abx^{2}\mathcal{F}_{0}\left( a,b,x\right) -ax\mathcal{F}_{1}\left(
a,b,x\right)%
\end{array}%
}{\left( \alpha +1\right) ^{n+1}\left( \beta +1\right) ^{n+1}}\left( \alpha
+1\right) ^{n+1}t^{2n}%
\end{array}%
\right\} .
\end{equation*}

Thus, we obtain%
\begin{eqnarray*}
G(t) &=&\frac{1}{\alpha -\beta }\overset{\infty }{\underset{n=0}{\sum }}%
\left( \frac{1}{abx^{2}}\right) ^{n+1}\left\{ 
\begin{array}{c}
-abx^{2}\left( \mathcal{F}_{1}\left( a,b,x\right) -bx\mathcal{F}_{0}\left(
a,b,x\right) \right) \beta ^{2n+1}\text{ \ \ \ \ \ \ \ \ } \\ 
-abx^{2}\left( bx\mathcal{F}_{0}\left( a,b,x\right) -\mathcal{F}_{1}\left(
a,b,x\right) \right) \alpha ^{2n+1}\text{ \ \ \ \ \ } \\ 
-bx\mathcal{F}_{0}\left( a,b,x\right) \beta ^{2n+2}+bx\mathcal{F}_{0}\left(
a,b,x\right) \alpha ^{2n+2}%
\end{array}%
\right\} t^{2n+1} \\
&&+\frac{1}{\alpha -\beta }\overset{\infty }{\underset{n=0}{\sum }}\left( 
\frac{1}{abx^{2}}\right) ^{n+1}\left\{ 
\begin{array}{c}
-abx^{2}\left( 
\begin{array}{c}
ax\mathcal{F}_{1}\left( a,b,x\right) -\mathcal{F}_{0}\left( a,b,x\right) \\ 
-abx^{2}\mathcal{F}_{0}\left( a,b,x\right)%
\end{array}%
\right) \beta ^{2n}\text{ \ \ \ \ \ \ } \\ 
-abx^{2}\left( 
\begin{array}{c}
abx^{2}\mathcal{F}_{0}\left( a,b,x\right) +\mathcal{F}_{0}\left( a,b,x\right)
\\ 
-ax\mathcal{F}_{1}\left( a,b,x\right)%
\end{array}%
\right) \alpha ^{2n} \\ 
-\mathcal{F}_{0}\left( a,b,x\right) \beta ^{2n+2}+\mathcal{F}_{0}\left(
a,b,x\right) \alpha ^{2n+2}%
\end{array}%
\right\} t^{2n}.
\end{eqnarray*}

Combining the sums, we get%
\begin{equation*}
G(t)=\overset{\infty }{\underset{n=0}{\sum }}\left\{ 
\begin{array}{c}
\left\{ 
\begin{array}{c}
\mathcal{F}_{1}\left( a,b,x\right) \\ 
\text{ \ \ \ }-bx\mathcal{F}_{0}\left( a,b,x\right)%
\end{array}%
\right\} ^{\varepsilon (n)}\left\{ 
\begin{array}{c}
-abx^{2}\mathcal{F}_{0}\left( a,b,x\right) \\ 
+ax\mathcal{F}_{1}\left( a,b,x\right) \\ 
-\mathcal{F}_{0}\left( a,b,x\right)%
\end{array}%
\right\} ^{1-\varepsilon (n)}\left( \frac{\alpha ^{n}-\beta ^{n}}{\left(
abx^{2}\right) ^{\left\lfloor \frac{n}{2}\right\rfloor }\left( \alpha -\beta
\right) }\right) \\ 
+\left( bx\right) ^{\varepsilon (n)}\mathcal{F}_{0}\left( a,b,x\right)
\left( \frac{\alpha ^{2\left( \left\lfloor \frac{n}{2}\right\rfloor
+1\right) }-\beta ^{2\left( \left\lfloor \frac{n}{2}\right\rfloor +1\right) }%
}{\left( abx^{2}\right) ^{\left\lfloor \frac{n}{2}\right\rfloor +1}\left(
\alpha -\beta \right) }\right)%
\end{array}%
\right\} t^{n}.
\end{equation*}

Therefore, for all $n\geq 0$, from the definition of generating function, we
have%
\begin{equation*}
\mathcal{F}_{n}\left( a,b,x\right) =A_{1}\left( \alpha ^{n}-\beta
^{n}\right) +B_{1}\left( \alpha ^{2\left\lfloor \frac{n}{2}\right\rfloor
+2}-\beta ^{2\left\lfloor \frac{n}{2}\right\rfloor +2}\right) ,
\end{equation*}

which is desired.
\end{proof}

Now, for bi-periodic Fibonacci matrix polynomial, we give the some \textit{%
summations by} considering its Binet formula.

\begin{corollary}
\label{teo2.5} For $k\geq 0$, the following statements are true:
\end{corollary}

\begin{itemize}
\item[$(i)$] 
\begin{equation}
\sum\limits_{k=0}^{n-1}\mathcal{F}_{k}\left( a,b,x\right) =\dfrac{%
\begin{array}{c}
a^{\varepsilon \left( n\right) }b^{1-\varepsilon \left( n\right) }\mathcal{F}%
_{n}\left( a,b,x\right) +a^{1-\varepsilon \left( n\right) }b^{\varepsilon
\left( n\right) }\mathcal{F}_{n-1}\left( a,b,x\right) \\ 
-a\mathcal{F}_{1}\left( a,b,x\right) +abx\mathcal{F}_{0}\left( a,b,x\right)
-b\mathcal{F}_{0}\left( a,b,x\right)%
\end{array}%
}{abx},  \notag
\end{equation}

\item[$(ii)$] 
\begin{equation}
\sum\limits_{k=0}^{n}\mathcal{F}_{k}\left( a,b,x\right) t^{-k}=\frac{1}{%
1-(ab+2)t^{2}+t^{4}}\left\{ 
\begin{array}{c}
\dfrac{\mathcal{F}_{n-1}\left( a,b,x\right) }{t^{n-1}}-\dfrac{\mathcal{F}%
_{n+1}\left( a,b,x\right) }{t^{n-3}} \\ 
+\dfrac{\mathcal{F}_{n}\left( a,b,x\right) }{t^{n}}-\dfrac{\mathcal{F}%
_{n+2}\left( a,b,x\right) }{t^{n+2}} \\ 
+t^{4}\mathcal{F}_{0}\left( a,b,x\right) +t^{3}\mathcal{F}_{1}\left(
a,b,x\right) \\ 
-t^{2}\left[ \left( abx^{2}+1\right) \mathcal{F}_{0}\left( a,b,x\right) -ax%
\mathcal{F}_{1}\left( a,b,x\right) \right] \\ 
-t\left( \mathcal{F}_{1}\left( a,b,x\right) -bx\mathcal{F}_{0}\left(
a,b,x\right) \right)%
\end{array}%
\right\} ,  \notag
\end{equation}

\item[$(iii)$] 
\begin{equation*}
\sum\limits_{k=0}^{\infty }\mathcal{F}_{k}\left( a,b,x\right) t^{-k}=\frac{t%
}{1-(abx^{2}+2)t^{2}+t^{4}}\left( 
\begin{array}{cc}
t^{3}+bxt^{2}-t & \frac{b}{a}t^{2}+bxt-\frac{b}{a} \\ 
t^{2}+axt-1 & t^{3}-\left( abx^{2}+1\right) t+bx%
\end{array}%
\right) ,
\end{equation*}%
where $\alpha ,\beta $ are roots of $r^{2}-abx^{2}r-abx^{2}=0$ equation and $%
\varepsilon (n)=n-2\left\lfloor \frac{n}{2}\right\rfloor $.
\end{itemize}

\section{Binomial transforms for Fibonacci matrix polynomial}

In this section, we mainly focus on the new matrix polynomial that obtained
by using the bi-periodic Fibonacci matrix polynomial.

\begin{definition}
\label{def1} For $n\in \mathbb{N}$, the matrix polynomial $\left( \mathcal{A}%
_{n}\left( a,b,x\right) \right) $ obtained by using bi-periodic Fibonacci
matrix polynomial is defined by%
\begin{equation}
\mathcal{A}_{n}\left( a,b,x\right) =\sqrt{x}a^{\frac{\varepsilon (n)}{2}}b^{%
\frac{1-\varepsilon (n)}{2}}\mathcal{F}_{n}\left( a,b,x\right)   \label{1}
\end{equation}%
where $a,b$ are nonzero real numbers and $\varepsilon (n)=n-2\left\lfloor 
\frac{n}{2}\right\rfloor .$
\end{definition}

\vskip0.4cm

In the following, we introduce the binomial transform and $k$-binomial
transform of the this matrix polynomial.

\begin{definition}
\label{def2} For $n\in \mathbb{N}$, the binomial and $k$-binomial transforms
of the matrix polynomial $\left( \mathcal{A}_{n}\left( a,b,x\right) \right) $
are defined by%
\begin{equation}
b_{n}\left( a,b,x\right) =\overset{n}{\underset{i=0}{\dsum }}\left( 
\begin{array}{c}
n \\ 
i%
\end{array}%
\right) \mathcal{A}_{i}\left( a,b,x\right) ,  \label{2}
\end{equation}%
\begin{equation}
w_{n}\left( a,b,x\right) =\overset{n}{\underset{i=0}{\dsum }}\left( 
\begin{array}{c}
n \\ 
i%
\end{array}%
\right) k^{n}\mathcal{A}_{i}\left( a,b,x\right) ,  \label{3}
\end{equation}%
respectively, where $a,b$ are nonzero real numbers$.$
\end{definition}

\vskip0.4cm

\begin{description}
\item Throughout this section, we will take $k=x\sqrt{ab}.$
\end{description}

\vskip0.4cm

Now, we give some properties for the binomial transform of the matrix
polynomial $\left( \mathcal{A}_{n}\left( a,b,x\right) \right) $.

\begin{theorem}
\label{teo1} The binomial transform of the matrix polynomial $\left( 
\mathcal{A}_{n}\left( a,b,x\right) \right) $ verifies the following
relations:
\end{theorem}

\begin{enumerate}
\item[$(i)$] $b_{n+1}\left( a,b,x\right) =\overset{n}{\underset{i=0}{\dsum }}%
\left( 
\begin{array}{c}
n \\ 
i%
\end{array}%
\right) \sqrt{x}a^{\frac{\varepsilon (n)}{2}}b^{\frac{1-\varepsilon (n)}{2}%
}\left( \mathcal{F}_{i}\left( a,b,x\right) +a^{\frac{1-2\varepsilon (i)}{2}%
}b^{\frac{2\varepsilon (i)-1}{2}}\mathcal{F}_{i+1}\left( a,b,x\right)
\right) ,$

\item[$(ii)$] $b_{n+1}\left( a,b,x\right) =\left( x\sqrt{ab}+2\right)
b_{n}\left( a,b,x\right) -x\sqrt{ab}b_{n-1}\left( a,b,x\right) ,$

\item[$(iii)$] $b_{n}\left( a,b,x\right) =C\left(
r_{2}^{n}(x)-r_{1}^{n}(x)\right) +\frac{\sqrt{bx}\mathcal{F}_{0}\left(
a,b,x\right) }{2}\left( r_{2}^{n}(x)+r_{1}^{n}(x)\right) ,$

\item[$(iv)$] $b_{n}\left( t\right) =\frac{\sqrt{bx}\mathcal{F}_{0}\left(
a,b,x\right) +t\left( \sqrt{ax}\mathcal{F}_{1}\left( a,b,x\right) -\sqrt{bx}%
\left( 1+x\sqrt{ab}\right) \mathcal{F}_{0}\left( a,b,x\right) \right) }{%
1-\left( x\sqrt{ab}+2\right) t+x\sqrt{ab}t^{2}},$

\item[$(v)$] $\overset{\infty }{\underset{n=0}{\dsum }}\frac{b_{n}\left(
a,b,x\right) t^{n}}{n!}=C\left( e^{r_{2}(x)t}-e^{r_{1}(x)t}\right) +\frac{%
\sqrt{bx}}{2}\mathcal{F}_{0}\left( a,b,x\right) \left(
e^{r_{2}(x)t}+e^{r_{1}(x)t}\right) ,$
\end{enumerate}

where $r_{1}(x)$, $r_{2}(x)$ are roots of the\ $r^{2}-\left( x\sqrt{ab}%
+2\right) r+x\sqrt{ab}=0$ equation and $C=\frac{bx\sqrt{ax}\mathcal{F}%
_{0}\left( a,b\right) -2\sqrt{ax}\mathcal{F}_{1}\left( a,b\right) }{2\sqrt{%
abx^{2}+4}}$.

\begin{proof}
We will prove the first two equalities because the proof of the others can
be done in similar ways.

$(i)$ By considering the property of binomial numbers, we can write%
\begin{eqnarray*}
b_{n+1}\left( a,b,x\right) &=&\overset{n+1}{\underset{i=0}{\dsum }}\left( 
\begin{array}{c}
n+1 \\ 
i%
\end{array}%
\right) \mathcal{A}_{i}\left( a,b,x\right) =\overset{n+1}{\underset{i=0}{%
\dsum }}\left( 
\begin{array}{c}
n+1 \\ 
i%
\end{array}%
\right) \sqrt{x}a^{\frac{\varepsilon (i)}{2}}b^{\frac{1-\varepsilon (i)}{2}}%
\mathcal{F}_{i}\left( a,b,x\right) \\
&=&\overset{n+1}{\underset{i=0}{\dsum }}\left[ \left( 
\begin{array}{c}
n \\ 
i%
\end{array}%
\right) +\left( 
\begin{array}{c}
n \\ 
i-1%
\end{array}%
\right) \right] \sqrt{x}a^{\frac{\varepsilon (i)}{2}}b^{\frac{1-\varepsilon
(i)}{2}}\mathcal{F}_{i}\left( a,b,x\right) .
\end{eqnarray*}

If necessary arrengements are made, we have%
\begin{eqnarray*}
b_{n+1}\left( a,b,x\right) &=&\overset{n}{\underset{i=0}{\dsum }}\left( 
\begin{array}{c}
n \\ 
i%
\end{array}%
\right) \sqrt{x}a^{\frac{\varepsilon (i)}{2}}b^{\frac{1-\varepsilon (i)}{2}}%
\mathcal{F}_{i}\left( a,b,x\right) +\overset{n+1}{\underset{i=0}{\dsum }}%
\left( 
\begin{array}{c}
n \\ 
i-1%
\end{array}%
\right) \sqrt{x}a^{\frac{\varepsilon (i)}{2}}b^{\frac{1-\varepsilon (i)}{2}}%
\mathcal{F}_{i}\left( a,b,x\right) \\
&=&\overset{n}{\underset{i=0}{\dsum }}\left( 
\begin{array}{c}
n \\ 
i%
\end{array}%
\right) \sqrt{x}a^{\frac{\varepsilon (i)}{2}}b^{\frac{1-\varepsilon (i)}{2}}%
\mathcal{F}_{i}\left( a,b,x\right) +\overset{n}{\underset{i=0}{\dsum }}%
\left( 
\begin{array}{c}
n \\ 
i%
\end{array}%
\right) \sqrt{x}a^{\frac{1-\varepsilon (i)}{2}}b^{\frac{\varepsilon (i)}{2}}%
\mathcal{F}_{i+1}\left( a,b,x\right) \\
&=&\overset{n}{\underset{i=0}{\dsum }}\left( 
\begin{array}{c}
n \\ 
i%
\end{array}%
\right) a^{\frac{\varepsilon (n)}{2}}b^{\frac{1-\varepsilon (n)}{2}}\left( 
\sqrt{x}\mathcal{F}_{i}\left( a,b,x\right) +\sqrt{x}a^{\frac{1-2\varepsilon
(i)}{2}}b^{\frac{2\varepsilon (i)-1}{2}}\mathcal{F}_{i+1}\left( a,b,x\right)
\right) .
\end{eqnarray*}

Also, we can write as $b_{n+1}\left( a,b,x\right) =b_{n}\left( a,b,x\right) +%
\overset{n}{\underset{i=0}{\dsum }}\left( 
\begin{array}{c}
n \\ 
i%
\end{array}%
\right) \mathcal{A}_{i+1}\left( a,b,x\right) .$

$(ii)$ By using the equation $(i),$ we can write%
\begin{eqnarray*}
b_{n+1}\left( a,b,x\right) &=&\overset{n}{\underset{i=1}{\dsum }}\left( 
\begin{array}{c}
n \\ 
i%
\end{array}%
\right) \sqrt{x}a^{\frac{\varepsilon (i)}{2}}b^{\frac{1-\varepsilon (i)}{2}%
}\left( \mathcal{F}_{i}\left( a,b,x\right) +a^{\frac{1-2\varepsilon (i)}{2}%
}b^{\frac{2\varepsilon (i)-1}{2}}\mathcal{F}_{i+1}\left( a,b,x\right) \right)
\\
&&+\sqrt{bx}\mathcal{F}_{0}\left( a,b,x\right) +\sqrt{ax}\mathcal{F}%
_{1}\left( a,b,x\right) \\
&=&\left( x\sqrt{ab}+1\right) \overset{n}{\underset{i=1}{\dsum }}\left( 
\begin{array}{c}
n \\ 
i%
\end{array}%
\right) \sqrt{x}a^{\frac{\varepsilon (i)}{2}}b^{\frac{1-\varepsilon (i)}{2}}%
\mathcal{F}_{i}\left( a,b,x\right) \\
&&+\overset{n}{\underset{i=1}{\dsum }}\left( 
\begin{array}{c}
n \\ 
i%
\end{array}%
\right) \sqrt{x}a^{\frac{1-\varepsilon (i)}{2}}b^{\frac{\varepsilon (i)}{2}}%
\mathcal{F}_{i-1}\left( a,b,x\right) +\sqrt{bx}\mathcal{F}_{0}\left(
a,b,x\right) +\sqrt{ax}\mathcal{F}_{1}\left( a,b,x\right) \\
&=&\left( x\sqrt{ab}+1\right) b_{n}\left( a,b,x\right) +\overset{n}{\underset%
{i=1}{\dsum }}\left( 
\begin{array}{c}
n \\ 
i%
\end{array}%
\right) \sqrt{x}a^{\frac{1-\varepsilon (i)}{2}}b^{\frac{\varepsilon (i)}{2}}%
\mathcal{F}_{i-1}\left( a,b,x\right) \\
&&-x\sqrt{ab}\mathcal{F}_{0}\left( a,b,x\right) +\sqrt{ax}\mathcal{F}%
_{1}\left( a,b,x\right)
\end{eqnarray*}

Thus, we obtain%
\begin{eqnarray*}
b_{n}\left( a,b,x\right) &=&\left( x\sqrt{ab}+1\right) b_{n-1}\left(
a,b,x\right) +\overset{n-1}{\underset{i=1}{\dsum }}\left( 
\begin{array}{c}
n-1 \\ 
i%
\end{array}%
\right) \sqrt{x}a^{\frac{1-\varepsilon (i)}{2}}b^{\frac{\varepsilon (i)}{2}}%
\mathcal{F}_{i-1}\left( a,b,x\right) \\
&&-x\sqrt{ab}\mathcal{F}_{0}\left( a,b,x\right) +\sqrt{ax}\mathcal{F}%
_{1}\left( a,b,x\right) \\
&=&x\sqrt{ab}b_{n-1}\left( a,b,x\right) +\overset{n-1}{\underset{i=0}{\dsum }%
}\left( 
\begin{array}{c}
n-1 \\ 
i%
\end{array}%
\right) \sqrt{x}a^{\frac{\varepsilon (i)}{2}}b^{\frac{1-\varepsilon (i)}{2}}%
\mathcal{F}_{i}\left( a,b,x\right) \\
&&+\overset{n-1}{\underset{i=1}{\dsum }}\left( 
\begin{array}{c}
n-1 \\ 
i%
\end{array}%
\right) \sqrt{x}a^{\frac{1-\varepsilon (i)}{2}}b^{\frac{\varepsilon (i)}{2}}%
\mathcal{F}_{i-1}\left( a,b,x\right) -x\sqrt{ab}\mathcal{F}_{0}\left(
a,b,x\right) +\sqrt{ax}\mathcal{F}_{1}\left( a,b,x\right) \\
&=&x\sqrt{ab}b_{n-1}\left( a,b,x\right) +\overset{n}{\underset{i=1}{\dsum }}%
\left( 
\begin{array}{c}
n \\ 
i%
\end{array}%
\right) \sqrt{x}a^{\frac{1-\varepsilon (i)}{2}}b^{\frac{\varepsilon (i)}{2}}%
\mathcal{F}_{i-1}\left( a,b,x\right) \\
&&-x\sqrt{ab}\mathcal{F}_{0}\left( a,b,x\right) +\sqrt{ax}\mathcal{F}%
_{1}\left( a,b,x\right)
\end{eqnarray*}

And we get%
\begin{equation*}
b_{n+1}\left( a,b,x\right) =\left( x\sqrt{ab}+1\right) b_{n}\left(
a,b,x\right) +b_{n}\left( a,b,x\right) -x\sqrt{ab}b_{n-1}\left( a,b,x\right)
.
\end{equation*}
\end{proof}

From the definition of binomial and k-binomial transform, we obtain $%
w_{n}\left( a,b,x\right) =k^{n}b_{n}\left( a,b,x\right) =\left( x\sqrt{ab}%
\right) ^{n}b_{n}\left( a,b,x\right) .$ Thus, for every $n\in 
\mathbb{N}
,$ in the following equalities are true.

\begin{itemize}
\item $w_{n+1}\left( a,b,x\right) =\left( abx^{2}+2x\sqrt{ab}\right)
w_{n}\left( a,b,x\right) -abx^{3}\sqrt{ab}w_{n-1}\left( a,b,x\right) ,$

\item $w_{n}\left( t\right) =\frac{\sqrt{bx}\mathcal{F}_{0}\left(
a,b,x\right) +t\left( ax\sqrt{bx}\mathcal{F}_{1}\left( a,b,x\right) -\left(
bx\sqrt{ax}+abx^{2}\sqrt{bx}\right) \mathcal{F}_{0}\left( a,b,x\right)
\right) }{1-\left( abx^{2}+2x\sqrt{ab}\right) t+abx^{3}\sqrt{ab}t^{2}},$

\item $w_{n}\left( a,b,x\right) =C\left( r_{4}^{n}(x)-r_{3}^{n}(x)\right) +%
\frac{\sqrt{bx}\mathcal{F}_{0}\left( a,b,x\right) }{2}\left(
r_{4}^{n}(x)+r_{3}^{n}(x)\right) ,$
\end{itemize}

where $r_{3}(x)$ and $r_{4}(x)$ are roots of $r^{2}-\left( 2x\sqrt{ab}%
+abx^{2}\right) r-x^{3}ab\sqrt{ab}=0$\ equation and $C=\frac{bx\sqrt{ax}%
\mathcal{F}_{0}\left( a,b,x\right) -2\sqrt{ax}\mathcal{F}_{1}\left(
a,b,x\right) }{2\sqrt{abx^{2}+4}}.$

Now, we introduce the rising $k$-binomial transform of the matrix polynomial 
$\left( \mathcal{A}_{n}\left( a,b,x\right) \right) $.

\begin{definition}
\label{def4} For $n\in \mathbb{N}$, the rising $k$-binomial transform of the
matrix polynomial $\left( \mathcal{A}_{n}\left( a,b,x\right) \right) $ is
defined by%
\begin{equation}
r_{n}\left( a,b,x\right) =\overset{n}{\underset{i=0}{\dsum }}\left( 
\begin{array}{c}
n \\ 
i%
\end{array}%
\right) k^{i}\mathcal{A}_{i}\left( a,b,x\right)   \label{4}
\end{equation}%
where $a,b$ are nonzero real numbers.
\end{definition}

\vskip0.4cm

\begin{theorem}
\label{teo2} For every $n\in 
\mathbb{N}
,$ the rising $k$-binomial transform of the matrix polynomial $\left( 
\mathcal{A}_{n}\left( a,b,x\right) \right) $ is the polynomial $\left\{ 
\mathcal{A}_{2n}\left( a,b,x\right) \right\} $, that is%
\begin{equation}
r_{n}\left( a,b,x\right) =\mathcal{A}_{2n}\left( a,b,x\right) .  \label{2.9}
\end{equation}
\end{theorem}

\begin{proof}
From the Theorem \ref{teo2.4}, we can write%
\begin{eqnarray*}
r_{n}\left( a,b,x\right) &=&\overset{n}{\underset{i=0}{\dsum }}\left( 
\begin{array}{c}
n \\ 
i%
\end{array}%
\right) \left( x\sqrt{ab}\right) ^{i}\sqrt{x}a^{\frac{\varepsilon (i)}{2}}b^{%
\frac{1-\varepsilon (i)}{2}}\mathcal{F}_{i}\left( a,b\right) \\
&=&\overset{n}{\underset{i=0}{\dsum }}\left( 
\begin{array}{c}
n \\ 
i%
\end{array}%
\right) \left( ab\right) ^{\frac{i}{2}}x^{i+\frac{1}{2}}a^{\frac{\varepsilon
(i)}{2}}b^{\frac{1-\varepsilon (i)}{2}}\left[ A_{1}\left( \alpha ^{i}-\beta
^{i}\right) +B_{1}\left( \alpha ^{2\left\lfloor \frac{i}{2}\right\rfloor
+2}-\beta ^{2\left\lfloor \frac{i}{2}\right\rfloor +2}\right) \right] .
\end{eqnarray*}

Consequently, making the necessarry arrangements, we have%
\begin{equation*}
=\sqrt{bx}\left[ A_{1}\left( \alpha ^{2n}-\beta ^{2n}\right) +B_{1}\left(
\alpha ^{2n+2}-\beta ^{2n+2}\right) \right] =\sqrt{bx}\mathcal{F}_{2n}\left(
a,b,x\right) =\mathcal{A}_{2n}\left( a,b,x\right) .
\end{equation*}
\end{proof}

\begin{theorem}
\label{teo3} For every $n\in 
\mathbb{N}
,$ the recurrence relation for rising $k$-binomial transform of the matrix
polynomial $\left( \mathcal{A}_{n}\left( a,b,x\right) \right) $,%
\begin{equation}
r_{n+1}\left( a,b,x\right) =\left( abx^{2}+2\right) r_{n}\left( a,b,x\right)
-r_{n-1}\left( a,b,x\right) ,  \label{2.10}
\end{equation}%
where $r_{0}=\sqrt{bx}\mathcal{F}_{0}\left( a,b,x\right) $ and $r_{1}=\sqrt{%
bx}\mathcal{F}_{0}\left( a,b,x\right) +ax\sqrt{bx}\mathcal{F}_{1}\left(
a,b,x\right) .$
\end{theorem}

\begin{proof}
For the matrix polynomial $\left( \mathcal{A}_{2n}\left( a,b,x\right)
\right) $, the following relation can be written%
\begin{eqnarray*}
\mathcal{A}_{2n+2}\left( a,b,x\right) &=&\sqrt{bx}\mathcal{F}_{2n+2}\left(
a,b,x\right) =\sqrt{bx}\left( ax\mathcal{F}_{2n+1}\left( a,b,x\right) +%
\mathcal{F}_{2n}\left( a,b,x\right) \right) \\
&=&\sqrt{bx}\left( ax\left( bx\mathcal{F}_{2n}\left( a,b,x\right) +\mathcal{F%
}_{2n-1}\left( a,b,x\right) \right) +\mathcal{F}_{2n}\left( a,b,x\right)
\right) \\
&=&\sqrt{bx}\left( \left( abx^{2}+1\right) \mathcal{F}_{2n}\left(
a,b,x\right) +\mathcal{F}_{2n}\left( a,b,x\right) -\mathcal{F}_{2n-2}\left(
a,b,x\right) \right) \\
&=&\left( abx^{2}+2\right) \left( \sqrt{bx}\mathcal{F}_{2n}\left(
a,b,x\right) \right) -\sqrt{bx}\mathcal{F}_{2n-2}\left( a,b,x\right) \\
&=&\left( abx^{2}+2\right) \mathcal{A}_{2n}\left( a,b,x\right) -\mathcal{A}%
_{2n-2}\left( a,b,x\right) .
\end{eqnarray*}

Therefore, from the Theorem \ref{teo2}, we find the desired result.
\end{proof}

In the following, we introduce the falling $k$-binomial transform of the
matrix polynomial $\left( \mathcal{A}_{n}\left( a,b,x\right) \right) $.

\begin{definition}
\label{def5} For $n\in \mathbb{N}$, the falling $k$-binomial transform of
the matrix polynomial $\left( \mathcal{A}_{n}\left( a,b,x\right) \right) $
is defined by%
\begin{equation}
f_{n}\left( a,b,x\right) =\overset{n}{\underset{i=0}{\dsum }}\left( 
\begin{array}{c}
n \\ 
i%
\end{array}%
\right) k^{n-i}\mathcal{A}_{i}\left( a,b,x\right)   \label{5}
\end{equation}%
where $a,b$ are nonzero real numbers$.$
\end{definition}

\vskip0.4cm

\begin{theorem}
\label{teo4} For every $n\in 
\mathbb{N}
,$ the recurrence relation for falling $k$-binomial transform of the matrix
polynomial $\left( \mathcal{A}_{n}\left( a,b,x\right) \right) $,%
\begin{equation*}
f_{n+1}\left( a,b,x\right) =3x\sqrt{ab}f_{n}\left( a,b,x\right) -\left(
2abx^{2}-1\right) f_{n-1}\left( a,b,x\right) ,
\end{equation*}%
where $f_{0}\left( a,b,x\right) =\sqrt{bx}\mathcal{F}_{0}\left( a,b,x\right) 
$ and $f_{1}\left( a,b,x\right) =bx\sqrt{ax}\mathcal{F}_{0}\left(
a,b,x\right) +\sqrt{ax}\mathcal{F}_{1}\left( a,b,x\right) .$
\end{theorem}

\begin{proof}
Firstly, we prove that%
\begin{eqnarray*}
f_{n+1}\left( a,b,x\right) &=&\overset{n}{\underset{i=0}{\dsum }}\left( 
\begin{array}{c}
n \\ 
i%
\end{array}%
\right) \left( ab\right) ^{\frac{n+1-i}{2}}x^{n+1-i}\left( \mathcal{A}%
_{i+1}\left( a,b,x\right) +\mathcal{A}_{i}\left( a,b,x\right) \right) \\
&=&\overset{n}{\underset{i=0}{\dsum }}\left( 
\begin{array}{c}
n \\ 
i%
\end{array}%
\right) \left( ab\right) ^{\frac{n+1-i}{2}}x^{n+1-i}\mathcal{A}_{i+1}\left(
a,b,x\right) +x\sqrt{ab}\mathcal{A}_{n}\left( a,b,x\right) .
\end{eqnarray*}

Thus, similar to the $(ii)$ of Theorem \ref{teo1}, the proof can be done.
\end{proof}

\begin{theorem}
\label{teo5} For every $n\in 
\mathbb{N}
,$ the Binet formula for falling and rising $k$-binomial transform of the
matrix polynomial $\left( \mathcal{A}_{n}\left( a,b,x\right) \right) $,%
\begin{equation*}
f_{n}\left( a,b,x\right) =C\left( r_{8}^{n}(x)-r_{7}^{n}(x)\right) +\frac{%
\sqrt{bx}\mathcal{F}_{0}\left( a,b,x\right) }{2}\left(
r_{8}^{n}(x)+r_{7}^{n}(x)\right) ,
\end{equation*}%
\begin{equation*}
r_{n}\left( a,b,x\right) =C\left( r_{6}^{n}(x)-r_{5}^{n}(x)\right) +\frac{%
\sqrt{bx}\mathcal{F}_{0}\left( a,b,x\right) }{2}\left(
r_{6}^{n}(x)+r_{5}^{n}(x)\right) ,
\end{equation*}%
where $r_{5}(x)$ and $r_{6}(x)$ are roots of $r^{2}-\left( abx^{2}+2\right)
r+1=0$ and $C=\frac{bx\sqrt{ax}\mathcal{F}_{0}\left( a,b,x\right) -2\sqrt{ax}%
\mathcal{F}_{1}\left( a,b,x\right) }{2\sqrt{abx^{2}+4}}$ also $r_{7}(x)$ and 
$r_{8}(x)$ are roots of $r^{2}-3x\sqrt{ab}r+2abx^{2}-1=0.$
\end{theorem}

\begin{proof}
Using the initial conditions, the Theorem can be proved.
\end{proof}

\section*{Conclusion}

\qquad In this paper, we define the matrix polynomial and give new
equalities for it. Then, defining the transforms for this matrix polynomial,
we get some properties of this transforms. Thus, it is obtained a new
genaralization for the polynomials, matrix sequences and number sequences
that have the similar recurrence relation in the literature. That is,

\begin{itemize}
\item If we take $a=b=1$ in Section 2, we get the some properties of the
Fibonacci polynomial.

\item If we take $a=b=2$ in Section 2, we get the some properties of the
Pell polynomial.

\item If we take $a=b=k$ in Section 2, we get the some properties of the $k$%
-Fibonacci polynomial.
\end{itemize}

If we choose $x=1$ in Section 3, then we obtain some properties for binomial
transforms of bi-periodic Fibonacci matrix sequence and bi-periodic
Fibonacci numbers.

Also, for different values of $a$ and $b,$ we obtain the some properties of
binomial transforms of the well-known matrix sequence and number sequence in
the literature:

\begin{itemize}
\item If we choose $a=b=1$, we obtain the some properties for binomial
transforms of Fibonacci matrix sequence and Fibonacci numbers.

\item If we choose $a=b=2$, we obtain the some properties for binomial
transforms of Pell matrix sequence and Pell numbers.

\item If we choose $a=b=k$, we obtain the some properties for binomial
transforms of $k$-Fibonacci matrix sequence and $k$-Fibonacci numbers.
\end{itemize}


\begin{thebibliography}{99}
\bibitem{Bilgici} Bilgici G., Two generalizations of Lucas sequence, \textit{%
Applied Mathematics and Computation}, \textbf{245} (2014), 526-538.

\bibitem{Civciv} Civciv H., T\"{u}rkmen R., On the $(s,t)$-Fibonacci and
Fibonacci matrix sequences, \textit{Ars Combinatoria}, \textbf{87} (2008),
161-173.

\bibitem{Coskun} Coskun A., Taskara N., The matrix sequence of bi-periodic
Fibonacci numbers, \textit{arXiv}, 1603.07487, (2016).

\bibitem{Edson} Edson M., Yayanie O., A new Generalization of Fibonacci
sequence and Extended Binet's Formula, \textit{Integers}, \textbf{9} (2009),
639-654.

\bibitem{FalconPlaza} Falcon S., Plaza A., Binomial Transforms of the $k$%
-Fibonacci Sequence, \textit{Int. J. of Nonlinear Sciences\&Numerical
Simulation}, \textbf{10} (2009), 1527-1538.

\bibitem{Plaza} Falcon S., Plaza A., On $k$-Fibonacci sequences and
polynomials and their derivatives, \textit{Chaos, Solitons \& Fractal}, 
\textbf{39} (2009), 1005-1019.

\bibitem{Hoggatt} Hoggatt Jr V. E., Bicknell M., Generalized Fibonacci
polynomials., \textit{Fibonacci Quart.}, \textbf{11}(5) (1973), 457-465.

\bibitem{Horadam} Horadam A.F., A generalized Fibonacci sequence, \textit{%
Math. Mag.,} \textbf{68} (1961), 455--459.

\bibitem{Koshy} Koshy T., \textit{Fibonacci and Lucas Numbers with
Applications}, John Wiley and Sons Inc, NY, 2001.

\bibitem{Ocal} Ocal A.A., Tuglu N., Altinisik E., On the representation of $%
k $-generalized Fibonacci and Lucas numbers, \textit{Applied Mathematics and
Computations}, \textbf{170} (1) (2005), 584--596.

\bibitem{Tasci} Tasci D., Firengiz M.C., Incomplete Fibonacci and Lucas $p$%
-numbers, \textit{Mathematical and Computer Modelling,} \textbf{52}(9)
(2010), 1763-1770.

\bibitem{Uslu} Uslu K., Uygun S., The $(s,t)$ Jacobsthal and $(s,t)$
Jacobsthal-Lucas Matrix Sequences, \textit{Ars Combinatoria},\textbf{\ 108}
(2013), 13-22.

\bibitem{YazlikTaskara} Yazlik Y., Taskara N., A note on generalized $k$%
-Horadam sequence, \textit{Computers and Mathematics with Applications}, 
\textbf{63} (2012), 36-41.

\bibitem{YilmazCoskunTaskara} Yilmaz N., Coskun A., Taskara N., On
properties of bi-periodic Fibonacci and Lucas Polynomials, \textit{ICNAAM
2016 (14th International Conference of Numerical Analysis and Applied
Mathematics 2016)}, 19-25 September, Rhodes, Greece, 2016.

\bibitem{YilmazTaskara} Yilmaz N., Taskara N., Binomial Transforms of the
Padovan and Perrin Matrix Sequences, \textit{Abstract and Applied Analysis,} 
\textbf{2013,} (2013).
\end{thebibliography}
\end{document}